\documentclass[11pt]{article}  

\usepackage{amsmath, amsthm, amsfonts, amssymb}
\usepackage[bookmarks]{hyperref}
\usepackage{indentfirst} 
\usepackage{marvosym}

\usepackage[a4paper]{geometry}

\newtheorem{theoremletters}{Theorem}

\newtheorem{corollaryletters}[theoremletters]{Corollary}

\newtheorem{propositionletters}[theoremletters]{Proposition}

\newtheorem{lemma}{Lemma}[section]
\newtheorem{theorem}[lemma]{Theorem}
\newtheorem{proposition}[lemma]{Proposition}
\newtheorem{corollary}[lemma]{Corollary}
\renewenvironment{proof}[1][\proofname]{{\sc #1. }}{\qed}
\theoremstyle{definition}
\newtheorem{remark}{Remark}
\newtheorem{example}{Example}
\newtheorem*{defin}{Definition}{\bf}{\rm}
\newtheorem*{fexam}{Final examples}{\bf}{\rm}

\newcommand{\abs}[1]{\ensuremath{\left| #1 \right|}}
\newcommand{\op}{\operatorname}
\newcommand{\ce}[2]{\operatorname{C}_{#1}(#2)}

\newcommand{\ze}[1]{\operatorname{Z}(#1)}

\newcommand{\fit}[1]{\operatorname{F}(#1)}
\newcommand{\rad}[2]{\op{O}_{#1}(#2)}
\newcommand{\syl}[2]{\op{Syl}_{#1}\left(#2\right)}
\newcommand{\hall}[2]{\op{Hall}_{#1}\left(#2\right)}

\begin{document}

\title{\bf Prime power indices in factorised groups}

\author{\sc M. J. Felipe $\cdot$ A. Mart\'inez-Pastor $\cdot$ V. M. Ortiz-Sotomayor
\thanks{The first author is supported by Proyecto Prometeo II/2015/011, Generalitat Valenciana (Spain), and the second author is supported by Proyecto MTM2014-54707-C3-1-P, Ministerio de Econom\'ia, Industria y Competitividad (Spain). The results in this paper are part of the third author's Ph.D. thesis, and he acknowledges the predoctoral grant ACIF/2016/170, Generalitat Valenciana (Spain). \newline
\rule{6cm}{0.1mm}\newline
Instituto Universitario de Matem\'atica Pura y Aplicada (IUMPA), Universitat Polit\`ecnica de Val\`encia, Camino de Vera, s/n, 46022, Valencia, Spain. \newline
\Letter: \texttt{mfelipe@mat.upv.es}, \texttt{anamarti@mat.upv.es}, \texttt{vicorso@doctor.upv.es}
}}

\date{}

\maketitle

\begin{abstract}
\noindent Let the group $G=AB$ be the product of the subgroups $A$ and $B$. We determine some structural properties of $G$ when the $p$-elements in $A\cup B$ have prime power indices in $G$, for some prime $p$. More generally, we also consider the case that all prime power order elements in $A\cup B$ have prime power indices in $G$. In particular, when $G=A=B$ we obtain as a consequence some known results. \\


\smallskip

\end{abstract}


\section{Introduction} 

Along this paper all groups considered are finite. Throughout the last decades, the impact of conjugacy class sizes (also called indices) over the structure of finite groups has been highly investigated. Simultaneously, several authors have studied groups factorised as the product of two subgroups, in particular when those factors verify certain relations of permutability (see \cite{BEA} for a detailed account on this subject). In this setting a main problem is how to infer structural properties from the factors to the whole group. The purpose of this paper is to present new achievements in the study of finite groups which combine both current research lines. Although the literature in this context is sparse, a first approach can be found either in \cite{BCL}, \cite{FMOsurvey}, or \cite{LWW}, where square-free class sizes were analysed. In this line, our concrete goal here is to obtain some structural facts about a factorised group, provided that the indices of certain prime power order elements in the factors are also prime powers. 

One of the usual troubles in the framework of conjugacy classes is that, a priori, it is not guaranteed that the indices of the elements in a subgroup divide the corresponding indices in the whole group. Surprisingly, under our hypotheses, we have been able to prove that this happens for the considered elements in the factors of a factorised group (see Proposition \ref{inheritstructure}). We also highlight that, in our development, we do not use any permutability property between the factors, in contrast to what occurs in the three above cited papers.

The origin of our research can be located in the manuscript of Baer \cite{B}, where the main result characterises all finite groups such that every prime power order element has prime power index. In 1990, Chillag and Herzog (\cite{CH}) analysed groups all of whose conjugacy classes have prime power size. Later on, these studies were enhanced by Camina and Camina in \cite{CC}. Instead of imposing the prime power index condition on all (prime power order) elements, they restricted focus only to those elements with order a $p$-number for a fixed prime $p$. Next, in 2005, Berkovich and Kazarin (\cite{BK}) addressed also several problems about prime power indices in finite groups. In particular, in both papers \cite{BK} and \cite{CC}, two alternative shorter proofs of the aforementioned Baer's characterisation are provided.

We will use the following terminology: for a group $G$ and an element $x\in G$, we call $i_G(x)$ the \emph{index} of $x$ in $G$, that is, $i_G(x)=\abs{G:\ce{G}{x}}$ is the size of the conjugacy class $x^G$. For a natural number $n$, we denote by $\pi(n)$ the set of prime divisors of $n$. In particular, $\pi(G)$ is the set of prime divisors of the order of $G$. If $p$ is a prime, then the set of all Sylow $p$-subgroups of $G$ is represented by $\syl{p}{G}$, and $\hall{\pi}{G}$ denotes the set of all Hall $\pi$-subgroups of $G$ for a set of primes $\pi$. A group such that $G=\rad{\pi}{G}\times\rad{\pi'}{G}$ is said to be \emph{$\pi$-decomposable}. Given a group $G=AB$ which is the product of the subgroups $A$ and $B$, a subgroup $S$ is called \emph{prefactorised} (with respect to this factorisation) if $S=(S\cap A)(S\cap B)$ (see \cite{AMB}). The remainder notation and terminology is standard in this topic, and it is taken mainly from \cite{DH}. We also refer to this book for details about classes of groups.

According to the paper of Camina and Camina \cite{CC}, given a group $G$ and a prime $p\in \pi(G)$, we call $G$ a \emph{$p$-Baer group} if every $p$-element has prime power index (hereafter, the natural number 1 is a power of every prime). Moreover, if each prime power order element has prime power index, $G$ is called a \emph{Baer group}. Inspired by those definitions, we introduce the following concepts for factorised groups:

\begin{defin}
Let $G=AB$ be the product of the subgroups $A$ and $B$, and let $p\in\pi(G)$. We say that:
\begin{itemize}
	\item[$\bullet$] $G=AB$ is a \textbf{$\pmb{p}$-Baer factorisation} if $i_G(x)$ is a prime power for every $p$-element $x\in A\cup B$;
	
	\item[$\bullet$] $G=AB$ is a \textbf{Baer factorisation} if $i_G(x)$ is a prime power for all prime power order elements $x\in A\cup B$, i.e., if it is a $p$-Baer factorisation for all $p$.
\end{itemize}
\end{defin}

Clearly, any central product of two ($p$-)Baer groups provides a ($p$-)Baer factorisation. 

Our first outcome is to determine structural information of a finite group $G$ which has a $p$-Baer factorisation:

\newpage

\begin{theoremletters}
\label{teop-baer}
Let $G=AB$ be a $p$-Baer factorisation, and let $P\in\syl{p}{G}$. Then:

(1) $G/\ce{G}{\rad{p}{G}}$ is $p$-decomposable.

(2) Both $P\fit{G}$ and $P\rad{p'}{G}$ are normal in $G$. In particular, $G$ is $p$-soluble of $p$-length 1.
	
(3) The Sylow $p$-subgroup of $G/\fit{G}$ is abelian.
	
(4) $P$ is abelian if and only if $\rad{p}{G}$ so is.
	
(5) If $P=(P\cap A)(P\cap B)$ and $P\cap X\nleqslant \ce{G}{\rad{p}{G}}$ for some $X\in\{A, B\}$, then $P\cap X$ centralises every Hall $p'$-subgroup of $G$. 
	
(6) If the Sylow $p$-subgroups of $A$ and $B$ are non-abelian, then $G$ is $p$-decomposable.

\end{theoremletters}

Besides, we get additional information based on the primes appearing as indices of the $p$-elements in the factors of a $p$-Baer factorisation:

\begin{theoremletters}
\label{teop-baerprimes}
Let $G=AB$ be a $p$-Baer factorisation, and let $P\in\syl{p}{G}$. Then there exist unique primes $q$ and $r$ such that $i_G(x)$ is a $q$-number for every $p$-element $x\in A$, and $i_G(y)$ is an $r$-number for every $p$-element $y\in B$, respectively. (Eventually $p\in\{q, r\}$ or $q=r$.) 

\noindent Moreover, $P\leqslant\ce{G}{\rad{\{q, r\}'}{\fit{G}}}$, and $P\rad{q}{G}\rad{r}{G}$ is normal in $G$. Further:

(1) If $q=r=p$, then $G$ is $p$-decomposable.
		
(2) If $p\notin \{q, r\}$, then $P$ is abelian.

\end{theoremletters}

In the particular case when $G=A=B$ in the above result and Theorem \ref{teop-baer} (2), we partially recover \cite[Theorem A]{CC} due to Camina and Camina (see Section \ref{sec_p}, Corollary \ref{theoCC}). 

Afterwards, we impose the prime power index condition on all prime power order elements in the factors (that is, we consider groups with a Baer factorisation). We start proving the main theorem of Baer's paper \cite{B} from our results on $p$-Baer factorisations when $G=A=B$ (see Theorem \ref{teoBAER}).

Then our first result for a non-trivial Baer factorisation is the next consequence of Theorem \ref{teop-baer}:

\begin{corollaryletters}
\label{teoNA}
If $G=AB$ is a Baer factorisation, then:

(1) $G/\fit{G}$ is abelian.
	
(2) $G$ has abelian Sylow subgroups (that is, $G$ is an A-group) if and only if $\fit{G}$ is abelian.
	
(3) Set $\sigma:=\{ p \in \pi(G) \: \mid \: A_p\in\syl{p}{A} \text{ and } B_p\in\syl{p}{B}\text{ are non-abelian}\}$. Then $G = \rad{\sigma}{G}\times\rad{\sigma'}{G}$ with $\rad{\sigma}{G}$ nilpotent. 
	
(4) If all Sylow subgroups of $A$ and $B$ are non-abelian, then $G$ is nilpotent.

\end{corollaryletters}

It is worthwhile to wonder whether the factors of a Baer factorisation are Baer groups. We have obtained that the answer is positive, in relation to the above comments on the divisibility of the indices:

\begin{propositionletters}
\label{inheritstructure}
Let $G=AB$ be a Baer factorisation. Let $x\in X$ be a prime power order element, where $X\in \{A, B\}$. If $i_G(x)$ is a $q$-number for some prime $q$, then $i_X(x)$ is also a $q$-number. In particular, it follows that $A$ and $B$ are Baer groups.
\end{propositionletters}

Consequently, the structure of $A$ and $B$ in a Baer factorisation $G=AB$ is well-known. Nevertheless, we cannot expect to get an analogous characterisation as Baer's one for Baer factorisations, even for direct products $G=A\times B$ (see Example \ref{counterBaer} (i)).

At best, some arithmetical and structural information about Baer factorisations arises locally, i.e., prime by prime:

\begin{theoremletters}
\label{teoallp}
Let $G=AB$ be a Baer factorisation. For a prime $p$, and given $P\in\syl{p}{G}$:

(1) If $P$ is not abelian, then $\abs{G:\ce{G}{P}}$ is a $\{p, q\}$-number, for a prime $q$. (Eventually, $p=q$.)
	
(2) If $P$ is abelian, then $\abs{G:\ce{G}{P}}$ is a $\{q, r\}$-number, for some primes $q$ and $r$, both distinct from $p$. (Eventually $q=r$.)

Further, $G/\ce{G}{\rad{p}{G}}$ is $p$-decomposable with abelian $p$-complement, and the $p$-complement has order divisible by at most two primes.
\end{theoremletters}

Finally, we have attained a characterisation of Baer factorisations through  the indices of the centralisers of the Sylow subgroups of the factors:

\begin{theoremletters}
\label{corindices}
Let $G=AB$ be the product of the subgroups $A$ and $B$. Then this is a Baer factorisation if and only if $\abs{G:\ce{G}{A_p}}$ and $\abs{G:\ce{G}{B_p}}$ are prime powers, for $A_p\in\syl{p}{A}$ and $B_p\in\syl{p}{B}$, and for every prime $p$.
\end{theoremletters}

In Section \ref{sec_p} we prove Theorems \ref{teop-baer} and \ref{teop-baerprimes}, which refer to prime power indices of $p$-elements, for a fixed prime $p$. The remaining stated results, which consider prime power order elements (for all primes), are proved in Section \ref{sec_allp}. We illustrate the scope of our research with some examples.


\section{Preliminary results}

We will use the following elementary properties frequently, sometimes without further reference.

\begin{lemma}
\label{inherithyp}
Let $N$ be a normal subgroup of a group $G$, and let $p$ be a prime. Then:

(a) $i_N(x)$ divides $i_G(x)$, for any $x\in N$.
	
(b) $i_{G/N}(xN)$ divides $i_G(x)$, for any $x\in G$.
	
(c)  If $xN$ is a $p$-element of $G/N$, then there exists a $p$-element $x_{1}\in G$ such that $xN = x_{1}N$.

\end{lemma}

The next result about Sylow subgroups of factorised groups will be useful along the paper. It is a convenient reformulation of \cite[1.3.3]{AMB}.

\begin{lemma}\emph{\cite[1.3.3]{AMB}}\label{prefact_sylow}
Let $G=AB$ be the product of the subgroups $A$ and $B$. Then for each $p\in \pi(G)$ there exists $P\in\syl{p}{G}$ such that $P= (P \cap A)(P \cap B)$, with $P \cap A\in\syl{p}{A}$ and $P \cap B\in\syl{p}{B}$.
\end{lemma}

\begin{remark}
We call attention to some facts on Sylow subgroups of factorised groups which will be used sometimes with no citation. Let $G=AB$ be the product of the subgroups $A$ and $B$, and let $p$ be a prime.

(1) Consider a Sylow $p$-subgroup $P=(P\cap A)(P\cap B)$ of $G$ such that $P\cap A\in\syl{p}{A}$ and $P\cap B\in\syl{p}{B}$. Then imposing arithmetical conditions on the indices of the $p$-elements in $A\cup B$ is equivalent to impose them on the indices of the elements in $(P\cap A)\cup (P\cap B)$, because of the conjugacy of Sylow $p$-subgroups.

(2) There exist easy examples which show that not every prefactorised Sylow $p$-subgroup $P=(P\cap A)(P\cap B)$ verifies that $P \cap A\in\syl{p}{A}$ and $P \cap B\in\syl{p}{B}$.

(3) In general, $\rad{p}{G}$ does not need to be prefactorised. However, if $P=(P\cap A)(P\cap B)\in\syl{p}{G}$ with either $P\cap A\leqslant\fit{G}$ or $P\cap B\leqslant\fit{G}$, then by the Dedekind law we get that in this case $\rad{p}{G}$ is prefactorised. 
\end{remark}

The following result is due to Wielandt.

\begin{lemma}\emph{\cite[Lemma 6]{B}}\label{wielandt}
Let $G$ be a finite group and $p\in \pi(G)$. If $x\in G$ is a $p$-element and $i_G(x)$ is a $p$-number, then $x\in\op{O}_{p}(G)$.
\end{lemma}

In \cite{CC}, Camina and Camina proved the next proposition, which extends the above lemma and a well-known result of Burnside about the non-simplicity of groups with a conjugacy class of prime power size.

\begin{proposition}\emph{\cite[Theorem 1]{CC}}
\label{CaminaCamina}
Let $G$ be a finite group. Then all elements of prime power index lie in $\op{F}_2(G)$, the second term of the Fitting series of $G$.
\end{proposition}

Finally, the lemma below due to Berkovich and Kazarin is a key fact in the proof of Theorem \ref{teop-baerprimes}.

\begin{lemma}\emph{\cite[Lemma 4]{BK}}
\label{lemmaBK}
Let $G$ be a finite group, and let $p$ be a prime. Suppose that the $p$-elements $x, y\in G\smallsetminus\ze{G}$ are such that $i_G(x)$ and $i_G(y)$ are powers of distinct primes, and that $i_G(xy)$ is also a power of a prime. Then $\langle x, y\rangle^G\leqslant\rad{p}{G}$ and $i_G(xy)=\op{max}\lbrace i_G(x), i_G(y)\rbrace$ is a power of $p$, so a Sylow $p$-subgroup of $G$ is non-abelian.
\end{lemma}


\section{Groups with a \texorpdfstring{$p$}{p}-Baer factorisation}
\label{sec_p}

In this section we will prove Theorems \ref{teop-baer} and \ref{teop-baerprimes} via a series of results. Firstly, we show two facts about $p$-decomposability in $p$-Baer factorisations.

\begin{lemma}
\label{lemma_p-decompo(1)}
Let $G=AB$ be the product of the subgroups $A$ and $B$, and let $p$ be a prime. Then $i_G(x)$ is a $p$-number for each $p$-element $x\in A\cup B$ if and only if $G$ is $p$-decomposable.
\end{lemma}

\begin{proof}
Only the necessity of the condition is in doubt. Let $P=(P\cap A)(P\cap B)\in\syl{p}{G}$, which exists by virtue of Lemma \ref{prefact_sylow}. The hypotheses and Lemma \ref{wielandt} lead to $x\in \op{O}_{p}(G)$, for every $x\in (P\cap A)\cup(P\cap B)$. It follows that $P$ is normal in $G$ and so $G=\rad{p}{G}H$, with $H$ a Hall $p'$-subgroup of $G$.

It remains only to prove that $[H, \op{O}_{p}(G)]=1$. We may assume $\rad{p}{G}\neq 1$, so there exists a minimal normal subgroup $N$ of $G$ such that $N\leqslant \rad{p}{G}$. Since the class of $p$-decomposable groups is a saturated formation and the hypotheses hold for quotients of $G$, it follows by induction on $\abs{G}$ that $N$ is the unique minimal normal subgroup of $G$ and $N=\fit{G}=\ce{G}{N}=\op{O}_{p}(G)$ (see \cite[A - 15.2, 15.8]{DH}). Consequently, since each element in $(\rad{p}{G}\cap A)\cup(\rad{p}{G}\cap B)$ has index a $p$-number and $\rad{p}{G}$ is abelian, it follows that all of them are central in $G$. This fact yields $\rad{p}{G}=(\rad{p}{G}\cap A)(\rad{p}{G}\cap B)\leqslant\ze{G}$ and the claim is proved.  
\end{proof}

\begin{corollary}
\label{lemma_p-decompo(2)}
Let $G=AB$ be a $p$-Baer factorisation. Then $G/\ce{G}{\rad{p}{G}}$ is $p$-decomposable.
\end{corollary}

\begin{proof}
We can assume $1\neq \rad{p}{G} \nleqslant \ze{G}$. Denote $\overline{G}:=G/\ce{G}{\rad{p}{G}}$. If $\overline{G}$ is a $p'$-group the result follows, so let $1\neq \overline{x}\in \overline{A}\cup \overline{B}$ be a $p$-element. Then we can consider a $p$-element $x\in A\cup B$ such that $\overline{x}=x\ce{G}{\rad{p}{G}}$, and $i_G(x)$ is a prime power. But since $x\notin \ce{G}{\rad{p}{G}}$, it follows that $i_G(x)$ is a power of $p$, and $i_{\overline{G}}(\overline{x})$ so is. Finally, the previous lemma applies. 
\end{proof}

\medskip

The lemma below provides the proof of Theorem \ref{teop-baer} (2).

\begin{lemma}
\label{p-solv}
Let $G=AB$ be a $p$-Baer factorisation, and let $P\in\syl{p}{G}$. Then:

(a) $P\fit{G}$ is normal in $G$.
	
(b) $P\rad{p'}{G}$ is normal in $G$. In particular, $G$ is $p$-soluble of $p$-length 1.

\end{lemma}

\begin{proof}
(a) Let $P=(P\cap A)(P\cap B)\in\syl{p}{G}$, which exists by virtue of Lemma \ref{prefact_sylow}. By Proposition \ref{CaminaCamina} and our assumptions, we have $x\in \op{F}_2(G)$ for every element $x\in (P\cap A)\cup (P\cap B)$. Therefore $P\leqslant\op{F}_2(G)$, so $P\fit{G}/\fit{G}=\rad{p}{G/\fit{G}}$ and $P\fit{G}$ is normal in $G$.

(b) We proceed by induction on $\abs{G}$. If $N:=\rad{p'}{G}=1$, then the result follows by (a). Hence we may assume $N\neq 1$. Since $\overline{G}:=G/N$ inherits the hypotheses by Lemma \ref{inherithyp}, then $\overline{P}\rad{p'}{\overline{G}}=PN/N$ is normal in $G/N$, and the claim is proved. 
\end{proof}

\medskip

Note that the existence of Hall $p'$-subgroups in $p$-Baer factorisations is guaranteed as a consequence of the $p$-solubility of such groups. Indeed, if $G=AB$ is a Baer factorisation (i.e., it is $p$-Baer for all $p$), then it follows that $G/\fit{G}$ is nilpotent. This fact will be strengthened later (see Proposition \ref{teoNA}).

\begin{proposition}
\label{lemmaPXabelian}
Let $G=AB$ be a $p$-Baer factorisation, and let $P=(P\cap A)(P\cap B)\in\syl{p}{G}$.

(a) If for some $X\in\{A, B\}$ it holds that $P\cap X\nleqslant \fit{G}$, then $P\cap X \leqslant \ce{G}{\op{O}_{p}(G)}$, $P\cap X$ is abelian, and $[P\cap A, P\cap B]=1$.
	
(b) If both $P\cap A\nleqslant \fit{G}$ and $P\cap B\nleqslant\fit{G}$, then $P$ is abelian.
	
(c) The Sylow $p$-subgroup of $G/\fit{G}$ is abelian.
	
(d) $P$ is abelian if and only if $\rad{p}{G}$ so is.
	
(e) If $P\cap X\nleqslant\ce{G}{\rad{p}{G}}$ for some $X\in\{A, B\}$, then $P\cap X\leqslant\ce{G}{H}$ for every $H\in\hall{p'}{G}$. In particular, this holds when $P\cap X$ is non-abelian.

\end{proposition}

\begin{proof}
(a) Let $x \in (P\cap X)\smallsetminus \rad{p}{G}$. Then $i_G(x)$ is a power of a prime $q\neq p$ by Lemma \ref{wielandt}, so $x\in\ce{G}{\rad{p}{G}}$. Now let $y\in \rad{p}{G}\cap X$. Then $xy \in (P\cap X) \smallsetminus \op{O}_{p}(G)$, so $xy\in\ce{G}{\rad{p}{G}}$ and $y\in\ce{G}{\rad{p}{G}}$. It follows $P\cap X\in\ce{G}{\rad{p}{G}}$ and the first claim is proved.

Now we show that $P\cap X$ is abelian. By Lemma \ref{p-solv} (b), we get $\rad{p'}{G}\neq 1$, and for all $g\in G$ it holds that $P^g \leqslant (P\rad{p'}{G})^g = P\rad{p'}{G}$, so there exists $t_g\in \rad{p'}{G}$ such that $P^g=P^{t_g}$. Let $x\in (P\cap X)\smallsetminus \rad{p}{G}$. By the assumptions and Lemma \ref{wielandt}, we have that $i_G(x)$ is a power of a prime $r\neq p$, so there exists $P^g=P^{t_g}\leqslant\ce{G}{x}$. Hence $x\in\ce{G}{P}\rad{p'}{G}\leqslant\ce{G}{P\cap X}\rad{p'}{G}$. On the other hand, if $x\in \rad{p}{G}\cap X$, then $x\in \ce{G}{P\cap X}$ by the first paragraph. It follows $P\cap X \leqslant \ce{G}{P\cap X}\rad{p'}{G}$, and $[P\cap X, P\cap X]\leqslant P\cap [P\cap X, \rad{p'}{G}][P\cap X, \ce{G}{P\cap X}] \leqslant P\cap \rad{p'}{G} =1.$

Finally, we prove that $[P\cap A, P\cap B]=1$. Take for instance $X=A$. If $P\cap B\leqslant \fit{G}$, then the claim is clear since $P\cap A\leqslant \ce{G}{\rad{p}{G}}$. We may assume $P\cap B\nleqslant\fit{G}$, and so $\rad{p}{G}\leqslant\ce{G}{P\cap B}$. Moreover, as in the previous paragraph, if $x\in (P\cap A)\smallsetminus\rad{p}{G}$, then $x\in \ce{G}{P\cap B}\rad{p'}{G}$. Hence $P\cap A\leqslant \ce{G}{P\cap B}\rad{p'}{G}$, and $[P\cap A, P\cap B]=1$.

(b) It is a direct consequence of (a).

(c) The uniqueness of the Sylow $p$-subgroup of $G/\fit{G}$ follows from Lemma \ref{p-solv} (a). If both $P\cap A\nleqslant \fit{G}$ and $P\cap B\nleqslant \fit{G}$, then $P$ is abelian by (b), and so is $P\fit{G}/\fit{G}$. Finally, if for instance $P\cap A\nleqslant\fit{G}$ and $P\cap B\leqslant \fit{G}$, then $P\cap A$ is abelian by (a) and $P\fit{G}/\fit{G}=(P\cap A)\fit{G}/\fit{G}$ is also abelian. 

(d) Let show that if $P$ is not abelian, then $\rad{p}{G}$ cannot be abelian. Applying (b), we can assume for instance that $P\cap A\leqslant \fit{G}$ and $P\cap B\nleqslant \fit{G}$. Then by (a) we have that $P\cap B$ is abelian and $[P\cap A, P\cap B]=1$. Therefore $P\cap A=\rad{p}{G} \cap A$ cannot be abelian, and so $\rad{p}{G}$ is not abelian either. 

(e) By (a) it holds that $P\cap X=\rad{p}{G}\cap X$. Let $x\in (P\cap X)\smallsetminus\ce{G}{\rad{p}{G}}$. Then $\rad{p}{G}\nleqslant \ce{G}{x}$, and $i_G(x)$ is a power of $p$. Since $x\in \rad{p}{G}$, then $\ce{G}{\rad{p}{G}}\leqslant \ce{G}{x}$, so $\abs{G/\ce{G}{\rad{p}{G}}:\ce{G}{x}/\ce{G}{\rad{p}{G}}} = \abs{G:\ce{G}{x}},$ which is a power of $p$. By Corollary \ref{lemma_p-decompo(2)}, $G/\ce{G}{\rad{p}{G}}$ is $p$-decomposable, and its unique Hall $p'$-subgroup is contained in $\ce{G}{x}/\ce{G}{\rad{p}{G}}$. Thus if $H\in\hall{p'}{G}$, we deduce $H^g\leqslant\ce{G}{x}$ for every $g\in G$, and for all $x\in (P\cap X)\smallsetminus\ce{G}{\rad{p}{G}}$. On the other hand, given $y \in P\cap X\cap \ce{G}{\rad{p}{G}}$, if $x \in (P\cap X)\smallsetminus\ce{G}{\rad{p}{G}}$, then $xy \in (P\cap X)\smallsetminus \ce{G}{\rad{p}{G}}$, so $H^g \leqslant \ce{G}{xy}\cap \ce{G}{x} = \ce{G}{x} \cap \ce{G}{y} \leqslant \ce{G}{y}. $ Therefore $H^g \leqslant \ce{G}{P\cap X}$ for all $g\in G$, and the first claim follows. 

Finally, if $P\cap X$ is non-abelian, then by (a) $P\cap X\leqslant\rad{p}{G}$ and so $P\cap X\nleqslant\ce{G}{\rad{p}{G}}$. 
\end{proof}

\begin{corollary}
\label{corSylAB}
Let $G=AB$ be a $p$-Baer factorisation. If the Sylow $p$-subgroups of $A$ and $B$ are non-abelian, then $G$ is $p$-decomposable.
\end{corollary}

\begin{proof}
It is sufficient to take $P=(P\cap A)(P\cap B)\in\syl{p}{G}$ with $P\cap A\in\syl{p}{A}$ and $P\cap B\in\syl{p}{B}$ and to apply the last statement of the above proposition. 
\end{proof}

\medskip

If we combine the previous results, we get the proof of Theorem \ref{teop-baer}.

\medskip

\begin{proof}[Proof of Theorem \ref{teop-baer}]
The statement (1) is exactly Corollary \ref{lemma_p-decompo(2)}. Lemma \ref{p-solv} yields (2), whilst claims (3), (4) and (5) are Proposition \ref{lemmaPXabelian} (c), (d) and (e), respectively. Corollary \ref{corSylAB} gives the last assertion. 
\end{proof}

\medskip

In the remainder of the section, we focus on proving Theorem \ref{teop-baerprimes}.

\begin{lemma}
\label{theoremPnotCent}
Let $G=AB$ be a $p$-Baer factorisation. Let $P=(P\cap A)(P\cap B)\in\syl{p}{G}$. If for some $X\in\{A, B\}$ it holds $P\nleqslant \ce{G}{\rad{p}{G}\cap X}$, then $\abs{G:\ce{G}{P\cap X}}$ is a $p$-number.
\end{lemma}

\begin{proof}
If $P\cap X\nleqslant \fit{G}$, then by Proposition \ref{lemmaPXabelian} (a) we get a contradiction. Therefore $P\cap X=\rad{p}{G}\cap X$. Since $P\nleqslant \ce{G}{\rad{p}{G} \cap X}$, we have that either $P \cap X$ is non-abelian or $[P\cap X, P\cap Y]\neq 1$ where $\{A, B\}=\{X, Y\}$. In the last case, $P\cap Y \leqslant \fit{G}$ by Proposition \ref{lemmaPXabelian} (a), and therefore $P\cap X \nleqslant \ce{G}{\rad{p}{G}}$. Thus we can apply in both cases Proposition \ref{lemmaPXabelian} (e) to deduce $P\cap X\leqslant\ce{G}{H}$, for every $H\in\hall{p'}{G}$. It follows that $\abs{G:\ce{G}{P\cap X}}$ is a power of $p$. 
\end{proof}

\begin{proposition}
\label{UniquePrime}
Let $G=AB$ be a $p$-Baer factorisation, and let $P\in\syl{p}{G}$. Then:

(a) There exist unique primes $q$ and $r$ such that $i_G(x)$ is a $q$-number for every $p$-element $x\in A$, and $i_G(y)$ is an $r$-number for every $p$-element $y\in B$, respectively. (Eventually $p\in\{q, r\}$ or $q=r$.) 
	
(b) $P\leqslant\ce{G}{\rad{\{q, r\}'}{\fit{G}}}$ and $P\rad{q}{G}\rad{r}{G}$ is normal in $G$.
	
(c) If $p\notin \{q, r\}$, then $P$ is abelian.

\end{proposition}

\begin{proof}
By Lemma \ref{prefact_sylow} we can assume that $P=(P\cap A)(P\cap B)$ where $P\cap A\in \syl{p}{A}$ and $P\cap B\in\syl{p}{B}$:

(a) We argue for instance with $A$. If the result is false, then there exist $p$-elements $a_1, a_2\in A$ such that $1\neq i_G(a_1)$ and $1\neq i_G(a_2)$ are relatively prime. By the conjugacy of the Sylow $p$-subgroups in $A$, we may assume $a_1, a_2 \in P\cap A$, and $i_G(a_1a_2)$ is a prime power. By Lemma \ref{lemmaBK}, $i_G(a_1a_2)=\op{max}\{i_G(a_1), i_G(a_2)\}$ is a $p$-number, and $a_1, a_2\in \rad{p}{G}\cap A$. Let assume that $\op{max}\{i_G(a_1), i_G(a_2)\}=i_G(a_1)$. Hence, since $a_1$ is not central, we get $P\nleqslant \ce{G}{\rad{p}{G}\cap A}$, and Lemma \ref{theoremPnotCent} leads to $a_2\in \ze{G}$, the final contradiction. 

(b) By (a) $\rad{\{q, r\}'}{\fit{G}}\leqslant\ce{G}{P\cap A}\cap\ce{G}{P\cap B}=\ce{G}{P}$. Applying Lemma \ref{p-solv} (a), we deduce that $P\rad{q}{G}\rad{r}{G}$ is normal in $G$.

(c) Let suppose that $P$ is not abelian. Then by Proposition \ref{lemmaPXabelian} (d), $P\cap X\nleqslant \ce{G}{\rad{p}{G}}$ for some $X\in\{A, B\}$. Finally, we deduce from Proposition \ref{lemmaPXabelian} (e) that $\abs{G:\ce{G}{P\cap X}}$ is a $p$-number, and so $p\in\{q, r\}$.
\end{proof}

\begin{example}
\label{examplePbaer}
The primes $q$ and $r$ in the previous result may not be equal. Let $G=A\times B$ be the direct product of a symmetric groups $A=\Sigma_3$ of three letters and a dihedral group $B=D_{10}$ of order ten, and consider the prime $p=2$. Clearly, that factorisation is $2$-Baer. Nevertheless, the $2$-elements $x\in A$ have $i_G(x)=3$ and the $2$-elements $y\in B$ have $i_G(y)=5$. Moreover, if $P\in\syl{2}{G}$, neither $P\rad{3}{G}$ nor $P\rad{5}{G}$ are normal in $G$.
\end{example}

Finally, we are ready to prove Theorem \ref{teop-baerprimes}.

\medskip

\begin{proof}[Proof of Theorem \ref{teop-baerprimes}]
The existence of the unique primes $q$ and $r$ follows from Proposition \ref{UniquePrime} (a). Then, the statement (1) is Lemma \ref{lemma_p-decompo(1)}, and the remaining assertions follow by Proposition \ref{UniquePrime} (b) and (c). 
\end{proof}

\medskip

If we take the trivial factorisation $G=A=B$ in Theorems \ref{teop-baer} (2) and \ref{teop-baerprimes}, we partially recover the main theorem of Camina and Camina in \cite{CC} about $p$-Baer groups.

\begin{corollary}\emph{\cite[Theorem A]{CC}}
\label{theoCC}
Let $G$ be a $p$-Baer group for some prime $p$. Then:

(a) $G$ is $p$-soluble with $p$-length 1, and
	
(b) there is a unique prime $q$ such that each $p$-element has $q$-power index. 
	
Further, let $Q\in\syl{q}{G}$, then
	
(c) if $p=q$, $P$ is a direct factor of $G$, or
	
(d) if $p\neq q$, $P$ is abelian, and $P\rad{q}{G}$ is normal in $G$.

\end{corollary}

Finally, we emphasize the relations between the primes appearing as indices of the $p$-elements in the factors of a $p$-Baer factorisation.

\begin{lemma}
\label{pqBaer}
Let $G=AB$ be a $p$-Baer factorisation. Let assume that there exist non-central $p$-elements $a\in A$ and $b\in B$, so that $i_G(a)$ is a $q$-number and $i_G(b)$ is an $r$-number, for some primes $q$ and $r$. 

Let assume that $q\neq p$ and the factorisation is also $q$-Baer. Then:

(a) If the $q$-elements in $A\cup B$ have indices an $s$-power, then $s\in \{p, r\}$. 
	
(b) Moreover, if $q=r$ and $s$ is the prime in (a), then $s=p$ and a Hall $\{p, q\}$-subgroup of $G$ is normal with abelian Sylow subgroups. 

\end{lemma}

\begin{proof}
(a) Take $P=(P\cap A)(P\cap B)\in\syl{p}{G}$ such that $P\cap A\in\syl{p}{A}$ and $P\cap B\in\syl{p}{B}$. We may assume that $a\in P\cap A$ and $b\in P\cap B$. Suppose that $s\neq p$, and we claim that $s=r$. If $s=q$, then by Lemma \ref{lemma_p-decompo(1)} we obtain that $G$ is $q$-decomposable, which contradicts that $1\neq i_G(a)$ is a $q$-number and $a$ is a $q'$-element. Hence, $s\notin\{p, q\}$. Now if we assume also that $s\neq r$, then $\pi(i_G(z))\cap\{p, q, r\}=\emptyset$ for any  $q$-element $z\in A\cup B$. Since $P\rad{q}{G}\rad{r}{G}$ is a normal $\{p, q, r\}$-subgroup of $G$ by Proposition \ref{UniquePrime} (b), given $Q=(Q\cap A)(Q\cap B)\in\syl{q}{G}$ it follows $P\cap A\leqslant P\rad{q}{G}\rad{r}{G}\leqslant \ce{G}{Q\cap A}\cap \ce{G}{Q\cap B}=\ce{G}{Q}.$ But this contradicts again that $i_G(a)\neq 1$ is a $q$-number.

(b) By (a), we deduce $s\in \{p, r\}$. As above, since $q=r$ we get $s=p$ because of Lemma \ref{lemma_p-decompo(1)}. As a consequence, Proposition  \ref{UniquePrime} (b) yields that $P\rad{q}{G}$ and $Q\rad{p}{G}$ are normal in $G$, for $P\in\syl{p}{G}$ and $Q\in \syl{q}{G}$. Hence $PQ\rad{q}{G}\rad{p}{G}=PQ\unlhd G,$ and it is a Hall $\{p, q\}$-subgroup of $G$. The abelianity of the Sylow subgroups of $PQ$ follows from Proposition \ref{UniquePrime} (c). 
\end{proof}

\medskip

If we choose the trivial factorisation $G=A=B$ in the above result, we recover:

\begin{corollary}\emph{\cite[Lemma 5]{CC}}
Let $G$ be a $p$-Baer group and a $q$-Baer group for primes $p\neq q$. Suppose that all $p$-elements have $q$-power index. Then all $q$-elements have $p$-power index.
\end{corollary}


\section{Groups with a Baer factorisation}
\label{sec_allp}

In the sequel, the prime power index condition is imposed on all prime power order elements in the factors, that is, we consider Baer factorisations. 

We start this section by proving Baer's theorem (\cite{B}) as a consequence of the results obtained in Section \ref{sec_p}, when we consider the trivial factorisation $G=A=B$.

\begin{theorem}\emph{\cite[Section 3 - Theorem]{B}}
\label{teoBAER}
Let $G$ be a finite group. Each element $x\in G$ of prime power order has prime power index if and only if $$G=G_1 \times G_2 \times \cdots \times G_r,$$ where $G_i$ and $G_j$ have relatively prime orders for $i\neq j$, and if $G_i$ is not of prime order, then $\abs{\pi(G_i)}=2$ and its Sylow subgroups are abelian.
\end{theorem}

\begin{proof}
The converse is clear. Let $P\in\syl{p}{G}$. If $P$ is not abelian, then $P$ is a direct factor of $G$ by Corollary \ref{corSylAB}. Therefore, all non-abelian Sylow subgroups of $G$ are direct factors of it.

Now suppose that $P$ is abelian and non-central in $G$. Hence there is a $p$-element $x\in G$ such that $1\neq i_G(x)$ is a $q$-number, for some prime $q\neq p$. Necessarily, by Lemma \ref{lemma_p-decompo(1)}, there is a $q$-element $y\in G$ such that $1\neq i_G(y)$ is a $q'$-number. For some $Q\in\syl{q}{G}$, Lemma \ref{pqBaer} (b) yields that $PQ$ is a normal Hall $\{p, q\}$-subgroup of $G$, and $Q$ is also abelian.

Take a prime $r\notin \{p, q\}$ and $R\in\syl{r}{G}$. We may assume that $R$ is abelian and non-central. Thus, for each element $z\in R$, we deduce $\pi(i_G(z))\cap \{p, q\}=\emptyset$ by virtue of Lemma \ref{pqBaer} (b) again. Consequently $PQ\leqslant\ce{G}{R}$. Since this is valid for all primes $r\notin\{p, q\}$, the $\{p, q\}$-decomposability of $G$ follows. The result is now established. 
\end{proof}

\begin{remark}
The results stated in Section \ref{sec_p} can be also used to give an alternative proof of \cite[Theorem 2]{CH} due to Chillag and Herzog, avoiding Theorem \ref{teoBAER}.
\end{remark}

\medskip

\begin{proof}[Proof of Corollary \ref{teoNA}]
We deduce the first two statements from a direct application of Theorem \ref{teop-baer} (3) and (4), respectively. The final two assertions follow from Theorem \ref{teop-baer} (6). 
\end{proof}

\medskip

Next we are proving that the factors of a Baer factorisation are Baer groups. This is because in such a factorisation $G=AB$, the prime power index condition is inherited by both factors, even if $A$ and $B$ are not subnormal in $G$. This is no longer true for other arithmetical conditions on the indices (see for instance \cite{FMOsurvey} for the square-free property). In particular, as pointed out in \cite{BK}, subgroups of Baer groups are also Baer groups.  It is an open question whether the factors of a $p$-Baer factorisation are $p$-Baer groups. Nevertheless, it might happen for such a group that the indices in a factor and in the whole group are powers of distinct primes (see Final examples (2)).

\medskip

\begin{proof}[Proof of Proposition \ref{inheritstructure}]
Let $P=(P\cap A)(P\cap B)\in\syl{p}{G}$ such that $P\cap A\in\syl{p}{A}$ and $P\cap B\in\syl{p}{B}$, for some prime $p$. Let $X\in\{A, B\}$ and take $x\in (P\cap X)\smallsetminus\ze{G}$ such that $i_G(x)$ is a $q$-number. If $q=p$, then Proposition \ref{UniquePrime} (a) and Lemma \ref{wielandt} yields $P\cap X=\rad{p}{G}\cap X$. Moreover, $P\nleqslant \ce{G}{\rad{p}{G}\cap X}$ because $1\neq i_G(x)$ is a $p$-power. Thus, by virtue of Lemma \ref{theoremPnotCent}, we deduce that $\abs{G:\ce{G}{P\cap X}}$ is a $p$-power. As $P\cap X$ is normal in $X$, then $\abs{X:\ce{X}{P\cap X}}=\abs{X\ce{G}{P\cap X}:\ce{G}{P\cap X}}$ divides $\abs{G:\ce{G}{P\cap X}}$. Therefore $i_X(x)$ divides the $p$-number $\abs{G:\ce{G}{P\cap X}}$.

Hence we may assume $q\neq p$. Suppose $\rad{q}{G}\neq 1$. Note that the quotient $\overline{G}:=G/\rad{q}{G}$ inherits the hypotheses. It follows by induction that $i_{\overline{X}}(\overline{x}) = \abs{\overline{X}:\ce{\overline{X}}{\overline{x}}}$ is a $q$-number, because $i_{\overline{G}}(\overline{x})$ divides $i_G(x)$. However, since $q\neq p$, applying \cite[3.2.8]{KS} and the isomorphism $\overline{X}\cong X/(X\cap \rad{q}{G})$ we deduce $\ce{\overline{X}}{\overline{x}}=\overline{\ce{X}{x}}$. Thus $\abs{\overline{X}:\ce{\overline{X}}{\overline{x}}}\cdot \abs{(X\cap\rad{q}{G}):(\ce{X}{x}\cap\rad{q}{G})}=\abs{X:\ce{X}{x}},$ which is also a $q$-power, and the result is proved in this case.  

Now we assume $\fit{G}=\rad{q'}{\fit{G}}$. Let $M:=X\fit{G}$ which is normal in $G$ by Corollary \ref{teoNA} (1). Then $M = X(M\cap Y)$ with $\{X, Y\}=\{A, B\}$, and $M$ verifies the hypotheses. If $M<G$, by induction we get that $i_X(x)$ is a power of the same prime that divides $i_M(x)$, which divides $i_G(x)$. Consequently we may assume $G=M=X\rad{q'}{\fit{G}}$. Let $G_{q'}\in\hall{q'}{G}$. Then $G_{q'}=\rad{q'}{\fit{G}}(X\cap G_{q'})$. Moreover, $\abs{G:G_{q'}} =\abs{X\rad{q'}{\fit{G}}:(X\cap G_{q'})\rad{q'}{\fit{G}}} = \abs{X:X\cap G_{q'}}$. Therefore, for each $G_{q'}\in\hall{q'}{G}$, we have that $X\cap G_{q'}$ is also a Hall $q'$-subgroup of $X$. Since $i_G(x)$ is a $q$-number, there exists some Hall $q'$-subgroup of $G$ that centralises $x$, and so there exists a Hall $q'$-subgroup of $X$ that centralises $x$, and we are done. 
\end{proof}

\begin{example}
\label{counterBaer}

(i) In contrast to Baer's theorem (Theorem \ref{teoBAER}), and in spite of the above proposition, in a Baer factorisation $G=AB$ it is not guaranteed that $G$ is a direct product of proper Hall subgroups  for pairwise disjoint sets of primes, even for direct products: To see this consider $A=C_3 \times [C_7]C_2 \times [C_{11}]C_5$ and $B=C_5 \times [C_7]C_3 \times [C_{11}]C_2$. Then $G=A\times B$ is a Baer factorisation, but there are no pairwise coprime proper direct factors of $G$.
	
(ii) We highlight that there are Baer factorisations which are not just a central product of Baer groups: Let $G=H \times K$ be the direct product of a symmetric group $H=\Sigma_3$ and a dihedral group $K=D_{10}$. Let $A$ be a Sylow $2$-subgroup of $K$, and let $B$ be the direct product of $H$ and the Sylow $5$-subgroup of $K$. Then $G=AB$ is a Baer factorisation. Note that there is a $2$-element $g \in G \smallsetminus (A \cup B)$ such that $i_G(g)=15$, so $G$ is not a Baer group.
\end{example}

Now, as a step to prove Theorem \ref{teoallp}, an application of Lemma \ref{lemma_p-decompo(1)} gives the next result.

\begin{corollary}
\label{lemma_p-decompoallp1}
Let $G=AB$ be the product of the subgroups $A$ and $B$, and let $p$ be a prime. Then $i_G(x)$ is a $p$-number for each prime power order element $x\in A\cup B$ if and only if $G=\op{O}_{p}(G) \times \op{O}_{p'}(G)$, and $\op{O}_{p'}(G)$ is abelian.
\end{corollary}

\begin{proof}
The sufficient condition is straightforward. The $p$-decomposability of $G$ follows directly from Lemma \ref{lemma_p-decompo(1)}. Finally, if we take a prime $q\neq p$ and a prefactorised Sylow $q$-subgroup $Q=(Q\cap A)(Q\cap B)$, then $i_{\op{O}_{p'}(G)}(x)=1$ for each element $x\in (Q\cap A)\cup (Q\cap B)$. Therefore $Q\leqslant\ze{G}$ for every $q\neq p$, and the result follows. 
\end{proof}

\medskip

Furthermore, from the previous corollary we get:

\begin{corollary}
\label{lemma_p-decompoallp2}
Let $G=AB$ be the Baer factorisation. Then for each prime $p$ we have that $G/\ce{G}{\rad{p}{G}}$ is $p$-decomposable with abelian $p$-complement.
\end{corollary}

In the remainder of the section, we focus on proving Theorem \ref{teoallp}.

\begin{proposition}
\label{theoremPCent}
Let $G=AB$ be a Baer factorisation, and let $P\in\syl{p}{G}$. If $P\leqslant \ce{G}{\rad{p}{G}\cap X}$ for some $X\in\{A, B\}$, then $\ce{G}{\rad{p}{G}\cap X}$ is normal in $G$ and $G/\ce{G}{\rad{p}{G}\cap X}$ is an abelian $q$-group, for a prime $q\neq p$.
\end{proposition}

\begin{proof}
We denote $\widetilde{G} := G/\ce{G}{\rad{p}{G}}$ and we have $\widetilde{G} =  \rad{p}{\widetilde{G}} \times \widetilde{H}$ by Corollary \ref{lemma_p-decompoallp2}, where $\widetilde{H}\in\op{Hall}_{p'}(\widetilde{G})$ is abelian for any $H\in\hall{p'}{G}$. Since $P\ce{G}{\rad{p}{G}}\leqslant \ce{G}{\rad{p}{G}\cap X}$, we get by the Dedekind law $$\ce{G}{\rad{p}{G} \cap X}/ \ce{G}{\rad{p}{G}} = \rad{p}{\widetilde{G}} \times \widetilde{H_0} \unlhd \rad{p}{\widetilde{G}} \times \widetilde{H} = \widetilde{G},$$ where $H_0 := H\cap \ce{G}{\rad{p}{G} \cap X}$, and thus $\ce{G}{\rad{p}{G} \cap X}$ is normal in $G$. 

Set $\overline{G}:= G/\ce{G}{\rad{p}{G} \cap X}$. Then $\overline{G}$ is a $p'$-group, and since it is isomorphic to a quotient of $\widetilde{G}$, it is abelian.  We may affirm that there exists an element $x\in(\rad{p}{G}\cap X)\smallsetminus \ze{G}$. Then $i_G(x)$ is a $q$-number for some prime $q\neq p$ (actually, this holds for every element in $\rad{p}{G}\cap X$). Moreover, $|\overline{G}: \overline{\ce{G}{x}}| = \abs{G:\ce{G}{x}}$ so $q$ divides $|\overline{G}|$.  Let suppose that there exists another prime $r\neq q$ such that $r$ divides $|\overline{G}|$. Since $|\overline{G}: \overline{\ce{G}{x}}| = \abs{G:\ce{G}{x}}$, it follows that the unique Sylow $r$-subgroup $\overline{R}$ of $\overline{G}$ is contained in $\overline{\ce{G}{x}}$. Hence $R\leqslant\ce{G}{x}$ for every $x\in \rad{p}{G}\cap X$, so $R\leqslant \ce{G}{\rad{p}{G}\cap X}$, which contradicts that $r$ divides $|\overline{G}|$. 
\end{proof}

\medskip

This last proposition is not longer true for $p$-Baer factorisations, as Final examples (1) shows. The next result is the last step to prove Theorem \ref{teoallp}.

\begin{proposition}
\label{theoremPAnotFitting}
Let $G=AB$ be a Baer factorisation, and let $P=(P\cap A)(P\cap B)\in\syl{p}{G}$. Let assume that $P\cap X\nleqslant \fit{G}$ for some $X\in\{A, B\}$. Then:

(a) $P\cap X\leqslant \ze{P}$.
	
(b) There exists a unique prime $q\neq p$ such that $P\cap X\nleqslant \ce{G}{\rad{q}{G}}$.
	
(c) $\abs{G:\ce{G}{P\cap X}}$ is a power of the prime $q$ in statement (b).

\noindent If, moreover, $P\cap Y\nleqslant \fit{G}$ where $\{X, Y\}=\{A, B\}$, then:

(d) $P$ is abelian and $\abs{G:\ce{G}{P}}$ is a $\{q, r\}$-number, with $p\notin \{q, r\}$, $q$ is the prime in (b), and $r$ is the unique prime such that $P\cap Y \nleqslant\ce{G}{\rad{r}{G}}$. (Eventually $q=r$.)

\end{proposition}

\begin{proof}
(a) This is exactly Proposition \ref{lemmaPXabelian} (a).

(b) By (a), for every $x\in P\cap X$ it hold that $i_G(x)$ is a $q$-power, for a fixed prime $q\neq p$. Then $\rad{q'}{\fit{G}}\leqslant\ce{G}{P\cap X}$. Finally, $\rad{q}{G}\nleqslant\ce{G}{P\cap X}$ since otherwise $P\cap X\leqslant\ce{G}{\fit{G}}\leqslant\fit{G}$, a contradiction.

(c) Take $T\in\hall{\{p, q\}'}{G}$ such that $PT$ is a $q$-complement of $G$. As $G/\fit{G}$ is abelian by Proposition \ref{teoNA} (a), then $L:=P\rad{q}{G}T$ is normal in $G$. Consequently, for every $x\in P\cap X$ we obtain that $i_L(x)$ is a $q$-power, and there exists $g\in L$ such that $T^g$ centralises $x$,  (actually by (a) we may assume $g\in \rad{q}{G}$). Set $K:=(P\cap X)\rad{q}{G}$. Hence $K\subseteq \cup_{g\in K} \left(\ce{K}{T}\rad{q}{G}\right)^g\subseteq K.$ It follows that $K=\ce{K}{T}\rad{q}{G}$ and $[P\cap X, T]\leqslant [K, T]=[\ce{K}{T}\rad{q}{G}, T] = [\rad{q}{G}, T]\leqslant \rad{q}{G}$. But $[P\cap X, T]\leqslant PT$, which is a $q'$-group. Then $[P\cap X, T]=1$, and $P$ also centralises $P\cap X$ by (a). The claim is now proved.

(d) $P$ is abelian by (a). Moreover, by (c), $\abs{G:\ce{G}{P\cap A}}$ is a $q$-number and $\abs{G:\ce{G}{P\cap B}}$ is an $r$-number for some primes $q$ and $r$. Note that $\ce{G}{P}=\ce{G}{P\cap A}\cap \ce{G}{P\cap B}$. If $q\neq r$, then $G=\ce{G}{P\cap A}\ce{G}{P\cap B}$, and $\abs{G:\ce{G}{P}}=\abs{G:\ce{G}{P\cap A}}\cdot\abs{G:\ce{G}{P\cap B}}$ is a $\{q, r\}$-number. Therefore we may assume that $q=r$, and so every $p$-element in $A\cup B$ has $q$-power index. Arguing analogously as in (c), for a $T\in\hall{\{p, q\}'}{G}$ such that $PT$ is a $q$-complement, we deduce that $P\rad{q}{G} = (P\cap A)(P\cap B)\rad{q}{G} \leqslant \ce{G}{T}\rad{q}{G}$, and $[P, T]\leqslant PT\cap \rad{q}{G}=1$. Then $T$ centralises $P$, and $P$ is abelian, so $\abs{G:\ce{G}{P}}$ is a $q$-power. The result is now established. 
\end{proof}

\medskip

\begin{proof}[Proof of Theorem \ref{teoallp}]
(1)  Let $P=(P\cap A)(P\cap B)\in\syl{p}{G}$ non-abelian. If $P=\rad{p}{G}$, then either $P\nleqslant \ce{G}{\rad{p}{G}\cap A}$ or $P\nleqslant \ce{G}{\rad{p}{G}\cap B}$. Assume for instance that $P\nleqslant \ce{G}{\rad{p}{G}\cap A}$. Then by Lemma \ref{theoremPnotCent}, we get that $\abs{G:\ce{G}{P\cap A}}$ is a $p$-number. Moreover, if $P\nleqslant \ce{G}{\rad{p}{G}\cap B}$, then $\abs{G:\ce{G}{P\cap B}}$ is also a $p$-number, and each element in $(P\cap A)\cup(P\cap B)$ has index a $p$-number. Therefore $G$ is $p$-decomposable by virtue of Lemma \ref{lemma_p-decompo(1)}, and $\abs{G:\ce{G}{P}}$ is clearly a $p$-power. On the other hand, if $P\leqslant \ce{G}{\rad{p}{G}\cap B}$, Proposition \ref{theoremPCent} yields that $\ce{G}{\rad{p}{G} \cap B}$ is normal in $G$ and the quotient $G/\ce{G}{\rad{p}{G} \cap B}$ is an abelian $q$-group ($q\neq p$). We deduce $G=\ce{G}{P\cap A}\ce{G}{P\cap B}$, and since $\ce{G}{P}=\ce{G}{P\cap A}\cap \ce{G}{P\cap B}$, then $\abs{G:\ce{G}{P}}=\abs{G:\ce{G}{P\cap A}}\cdot\abs{G:\ce{G}{P\cap B}}$ is a $\{p, q\}$-number. 

Now we assume $P\nleqslant \fit{G}$. It cannot happen that both $P\cap A\nleqslant\fit{G}$ and $P\cap B\nleqslant\fit{G}$ by Proposition \ref{theoremPAnotFitting} (a). We may assume for instance that $P\cap A\nleqslant \fit{G}$ and $P\cap B\leqslant\fit{G}$. By Proposition \ref{theoremPAnotFitting} (a) and (c), we get that $P\cap A\leqslant\ze{P}$ and $\abs{G:\ce{G}{P\cap A}}$ is a $q$-number, where $q\neq p$. On the other hand, as $P$ is not abelian, $P\nleqslant \ce{G}{P\cap B}$, so Lemma \ref{theoremPnotCent} yields $\abs{G:\ce{G}{P\cap B}}$ is a $p$-number. Then $\abs{G:\ce{G}{P}}$ is a $\{p, q\}$-number.

(2) Consider that $P=(P\cap A)(P\cap B)\in\syl{p}{G}$ is abelian. If $P\cap A\nleqslant\fit{G}$ and $P\cap B\nleqslant\fit{G}$ then the claim follows from Proposition \ref{theoremPAnotFitting} (d). Assume that $P\cap A\nleqslant\fit{G}$ and $P\cap B\leqslant\fit{G}$, so $\abs{G:\ce{G}{P\cap A}}$ is a $q$-number with $q\neq p$ by Proposition \ref{theoremPAnotFitting} (c). Since $P\leqslant\ce{G}{\rad{p}{G}\cap B}$, we deduce from Proposition \ref{theoremPCent} that $\ce{G}{P\cap B}$ is normal in $G$ with index an $r$-number ($r\neq p$). If $q\neq r$, then $G=\ce{G}{P\cap A}\ce{G}{P\cap B}$ and the claim follows. If $q=r$, we obtain that $\abs{\ce{G}{P\cap A}:\ce{G}{P}} = \abs{\ce{G}{P\cap A}\ce{G}{P\cap B}:\ce{G}{P\cap B}}$ divides the $q$-number $\abs{G:\ce{G}{P\cap B}}$. Hence the index $\abs{G:\ce{G}{P}}$ is a $q$-power.

Now suppose $P=\rad{p}{G}$. Note that, by Proposition \ref{theoremPCent}, both $\ce{G}{\rad{p}{G}\cap A}$ and $\ce{G}{\rad{p}{G}\cap B}$ are normal in $G$ with indices a $q$-number and an $r$-number, respectively. The case $q\neq r$ is again clear. If $q=r$, the above reasoning on the index of $\abs{\ce{G}{P\cap A}:\ce{G}{P}}$ shows that $\abs{G:\ce{G}{P}}$ is a $q$-power.  Finally, from (1), (2) and Corollary \ref{lemma_p-decompoallp2} we conclude we last assertion of the theorem.
\end{proof}

\medskip

Now we end by proving Theorem \ref{corindices}, as a consequence of the previous result.

\medskip

\begin{proof}[Proof of Theorem \ref{corindices}]
The converse direction is trivial. Assume that $G=AB$ is a Baer factorisation, and let $P=(P\cap A)(P\cap B)\in\syl{p}{G}$ such that $P\cap A\in\syl{p}{A}$ and $P\cap B\in\syl{p}{B}$. Let $X\in\{A, B\}$. We claim that $\abs{G:\ce{G}{P\cap X}}$ is a prime power. Now we distinguish two cases: either $P\cap X\nleqslant \fit{G}$ or $P\cap X=\rad{p}{G}\cap X$. In the first case, the claim follows from Proposition \ref{theoremPAnotFitting} (c). In the second case, if $P\leqslant \ce{G}{\rad{p}{G}\cap X}$, then we apply Proposition \ref{theoremPCent}, and if $P\nleqslant \ce{G}{\rad{p}{G}\cap X}$, then Lemma \ref{theoremPnotCent} follows. 
\end{proof}

\begin{fexam}
Some of the results stated in this section fail when the hypotheses are weakened to $p$-Baer factorisations, as the following examples show:

(1) Proposition \ref{theoremPCent} and Theorem \ref{corindices} do not hold for $p$-Baer factorisations: Let $G=AB$ be the semidirect product of a non-abelian group $B$ of order 21 acting on an elementary abelian group $A$ of order 8, in such a way that the subgroup of order 7 permutes the involutions transitively (this group appears in \cite{CSS}). Then $i_G(g)=7$ for every $2$-element $g\in G$ (i.e., $G$ is $2$-Baer), and the unique abelian Sylow $2$-subgroup $P$ of $G$ is equal to $A$, but $\abs{G:\ce{G}{P}}=\abs{G:\ce{G}{\rad{2}{G}\cap A}}=21$.
	
(2) Proposition \ref{inheritstructure} does not hold either for $p$-Baer factorisations: Let $G$ be the group in (1). Then there is a subgroup $H$ of $G$ of order $24$ and a subgroup $K$ of order $7$ such that $G=HK$ is also a $2$-Baer factorisation, and there exists a $2$-element $x\in H$ such that $i_H(x)=3$, which clearly is not a $7$-number. 
	
(3) Proposition \ref{theoremPAnotFitting} (c) is neither true for $p$-Baer factorisations: Let now $Q$ be a cyclic group of order 7. Consider the regular wreath product $T=Q \op{wr} G$ with $G$ the group in (1), and denote by $Q^{\natural}$ the basis group (we point out that this group appears in \cite{BK}). Then the factorisation $T=Q^{\natural}G$ is $2$-Baer. Let $P$ be the Sylow $2$-subgroup of $G$, so $P\in\syl{2}{T}$. Then $P=P\cap G\nleqslant\fit{T}$, but the index $\abs{T:\ce{T}{P\cap G}}$ is divisible by $3$ and $7$.
\end{fexam}


\end{document}